\documentclass{amsart}
\usepackage{amsfonts}
\usepackage{color}

\def\R{{\mathbb {R}}}
\def\N{{\mathbb {N}}}

\def\L{{\mathcal{L}}}

\def\A{{\mathcal{A}}}

\def\D{{\mathcal{D}}}
\def\J{{\mathcal{J}}}

\def\O{{\Omega}}
\def\lam{\lambda}
\def\vp{\varphi}

\def\ve{\varepsilon}
\def\cd{\rightharpoonup}
\def\pp{|u(x)-u(y)|^{p-2}}

\def\div{\operatorname {\text{div}}}
\def\curl{\operatorname {\text{curl}}}

\newtheorem{teo}{Theorem}[section]
\newtheorem{lema}[teo]{Lemma}
\newtheorem{prop}[teo]{Proposition}
\newtheorem{corol}[teo]{Corollary}

\theoremstyle{remark}
\newtheorem{remark}[teo]{Remark}

\theoremstyle{definition}
\newtheorem{defi}[teo]{Definition}

\numberwithin{equation}{section}

\begin{document}

\title[$H-$convergence for nonlocal elliptic-type problems]{$H-$convergence result for nonlocal elliptic-type problems via Tartar's method}
\author[J. Fern\'andez Bonder, A. Ritorto and A.M. Salort]{Juli\'an Fern\'andez Bonder, Antonella Ritorto and Ariel Martin Salort}

\address{Departamento de Matem\'atica, FCEN -- Universidad de Buenos Aires and IMAS -- CONICET, Buenos Aires, Argentina}
\email[J. Fern\'andez Bonder]{jfbonder@dm.uba.ar}
\urladdr[J. Fern\'andez Bonder]{http://mate.dm.uba.ar/~jfbonder}

\email[A. Ritorto]{aritorto@dm.uba.ar}

\email[A.M. Salort]{asalort@dm.uba.ar}
\urladdr[A.M. Salort]{http://mate.dm.uba.ar/~asalort}



\begin{abstract}
In this work we obtain a compactness result for the $H-$convergence of a family of nonlocal and nonlinear monotone elliptic-type problems by means of Tartar's method of oscillating test functions.
\end{abstract}

\maketitle

\section{Introduction}

Homogenization theory dates back to the works of S. Spagnolo \cite{Sp68}, E. De Giorgi and S. Spagnolo \cite{SD73}, I. Babu{\v{s}}ka \cite{Ba76}, A. Bensoussan, J.L. Lions and G. Papanicolaou \cite{BLP78} and E. S\'anchez-Palencia \cite{SP80} among others. In the context of linear elliptic partial differential equations, the model to be studied is the limit as $n\to\infty$ of the following problems
\begin{align} \label{ec.i.1}
\begin{cases}
-\div(A_n \nabla u_n) = f &\text{ in } \O\\
u_n = 0 & \text{ on }\partial \O,
\end{cases}
\end{align}
where $\O\subset \R^N$ is a bounded domain, $f\in H^{-1}(\O)$ and $\{A_n\}_{n\in\N}\subset [L^\infty(\O)]^{N\times N}$ is a sequence of symmetric and uniformly coercive matrices.

As a model example, the authors considered the case where the matrices $A_n$ are given in terms of a single matrix $A$ in the form
$$
A_n(x) = A(nx),
$$
where $A$ is periodic, of period 1, in each variable.

In the periodic setting, the limit problem when $n\to\infty$ can easily be fully characterized. See \cite{BLP78}. 

In order to deal with the general case, Spagnolo and De Giorgi introduced the concept of $G-$convergence, that was later generalized by Murat and Tartar in the late 70s and is now called $H-$convergence. See \cite{CioDo}.

When F. Murat in 1974  was studying the behavior of \eqref{ec.i.1} as $n\to\infty$, one of the main drawbacks he found was the fact that two weakly convergent sequences do not converge, in general, to the product of their limits. Murat overcame this difficulty by developing a compensated compactness argument known as the {\em div-curl Lemma}, denomination suggested by his advisor, J.L. Lions,  due to the fact that  it results from a compensation effect. The Lemma was published  in 1978 \cite{Mu78} and an alternative proof was provided by L. Tartar also in 1978 \cite{T78} by using H\"ormander's compactness argument for the injection of $H^1_0(\O)$ into $L^2(\O)$. The lemma claims that if we consider two sequences $\{\psi_n\}_{n\in\N}$ and $\{\phi_n\}_{n\in\N}$ in $[L^2(\O)]^N$ such that
$$
\psi_n \cd \psi, \quad \text{ and } \quad \phi_n \cd \phi\quad \text{weakly  in } [L^2(\O)]^N,
$$
with the additional assumption that
$$
\div \psi_n \to \div \psi \text{ in } H^{-1}(\O), \quad \text{ and } \quad  \curl \phi_n \to \curl \phi \mbox{ in } [H^{-1}(\O)]^{N\times N},
$$
then we can guarantee that $\psi_n\cdot \phi_n \to \psi\cdot \phi$ in the sense of distributions. Recall that the curl of a vector field $\phi\in [L^2(\O)]^N$ is defined as
$$
\curl \phi = \left(\frac{\partial \phi^i}{\partial x^j} - \frac{\partial \phi^j}{\partial x^i}\right)_{1\le i,j\le N}.
$$

The div-curl Lemma plays a crucial role in homogenization theory. In fact, based on this lemma, Tartar introduced in \cite{T78,T77} a method leading to the limiting behavior of \eqref{ec.i.1} as $n \to \infty$, obtaining the existence of a coercive matrix $A_0\in [L^\infty(\O)]^{N\times N}$ such that the sequence of solutions $\{u_n\}_{n\in\N}$ of \eqref{ec.i.1} converges weakly in $H^1_0(\O)$, up to some subsequence, to a function $u_0$ which is the solution of the following {\em homogenized} limit problem
\begin{align} \label{ec.i.2}
\begin{cases}
-\div(A_0 u_0) = f & \text{ in } \O\\
u_0 = 0 & \text{ on }\partial \O.
\end{cases}
\end{align}	
Moreover, $A_n\nabla u_n \cdot \nabla u_n \to A_0\nabla u_0 \cdot \nabla u_0$ in the sense of distributions, see for instance \cite{Al,CioDo}. That is, the sequence $A_n$ $H-$converges to $A_0$.

In the quasilinear case, this type of results were obtained by several authors in the late 80s and the beginning of the 90s. We refer the interested reader to \cite{ChiadoPiat-DalMaso-Defranceschi, Pankov} and to G. Dal Maso's book \cite{DalMaso} where the authors use $\Gamma-$convergence methods in order to deal with these problems. See \cite{Braides-ChiadoPiat-Defranceschi} for the periodic case. Let us mentioned that $\Gamma-$convergence studies the behavior of minima in variational problems, so when specialized in quadratic functionals, this gives the behavior for symmetric elliptic problems.

We remark that in the linear case, $H-$convergence and $\Gamma-$convergence where recently shown to coincide even in the non symmetric case by Ansini, Dal Maso and Zeppieri \cite{Ansini-DalMaso-Zeppieri}.

More general classes of problems were addressed recently. In the case of periodic homogenization of certain Hamilton-Jacobi and fully nonlinear elliptic partial differential equations was studied first by Evans \cite{Ev92}. In the context of fully nonlinear uniformly elliptic equations in stationary ergodic media, the problem was studied by Caffarelli, Sounganidis and Wang \cite{CSW05}. In these papers the existence of homogenized equations is proved, but, due to the generality of these problems, no further information about the  structure of the limit problems was obtained.

Our intention in this work is to address the $H-$convergence problem to the nonlocal version of \eqref{ec.i.1} and to give a characterization of the homogenized limit problem. Before introducing our results, we  review the background regarding nonlocal problems and its homogenization.

In recent years, there has been a plenty of works on anomalous diffusion  where the standard Laplace operator,  which gives an explanation in terms of Brownian motion, has been replaced by nonlocal operators. The main aim was to extend the diffusion theory by taking into account the long range interactions. Such operators do not act by point-wise differentiation  but by a global integration with respect to a singular kernel. One prototype to have in mind is the so-called fractional laplacian defined by
\begin{equation}\label{frac.lap}
(-\Delta^s) u(x) := \text{p.v.} \int_{\R^N} \frac{u(x)-u(y)}{|x-y|^{N+2s}}\, dy, \qquad s\in(0,1),
\end{equation}
up to a normalization constant. The interest in studying this operator has a long history in probability since it is the infinitesimal generator of stable L\'evy processes. See \cite{AFT95,App04, Bertoin, CST91} and references therein.

For a general introduction to the mathematical analysis of these models, we refer the reader to the recent book \cite{Demengel}, the articles \cite{Bucur-Valdinoci, NPV12} and references therein.

The regularity theory for fully nonlinear integro-differential equations, which include the fractional laplacian as a trivial example, was recently extensively studied. See, for instance, \cite{CS09,CS11,RoS,Si06}.

Based in these regularity results for fully nonlinear integro-differential equations, R. Schwab in  \cite{Schwab1, Schwab2} extended the results of Evans and Caffarelli, Souganidis and Wang to this setting, but again no information on the limit problem is obtained. We recall that the results of Schwab make extensive use either of the periodicity or the ergodicity of the problem and the author does not obtain any general convergence result.

Aimed at obtaining more precise information on the homogenized equation and without making any assumptions on the behavior of the sequence of the operators, we focus our analysis to a general family of nonlinear anisotropic operators of the form 
\begin{equation}\label{L_A}
\L_a u(x) := \text{p.v.} \int_{\R^N} a(x,y) \frac{|u(x)-u(y)|^{p-2}(u(x)-u(y))}{|x-y|^{N+sp}}\, dy,\quad s\in (0,1),
\end{equation}
for a given positive and bounded kernel $a(x,y)$, where $p\in (1,\infty)$ is fixed.

Then the problem we address is the behavior as $n\to\infty$ of
\begin{align}  \label{ec.ii.1}
\begin{cases}
\L_{a_n} u_n =f & \text{ in } \O\\
u_n = 0 & \text{ in }\R^N \setminus \O,
\end{cases}
\end{align}
where $\O\subset\R^N$ is a bounded domain, $f\in L^{p'}(\O)$, $\frac{1}{p}+\frac{1}{p'}=1$, and $\{a_n\}_{n\in\N}$ denotes a sequence of uniformly bounded and positive kernels.

So in order to apply Tartar's method, we first prove a nonlocal version of the div-curl Lemma that allows us to deal with \eqref{ec.ii.1} as $n\to\infty$, leading to the limit  problem
\begin{align*} 
\begin{cases}
\L_{a_0} u_0 = f &\quad\text{ in } \O\\
u_0 = 0 &\quad \text{ in }\R^n \setminus \O.
\end{cases}
\end{align*}
The homogenized kernel $a_0(x,y)$ inherits the positivity and boundedness of the sequence $a_n(x,y)$. 

Finally, we show that this convergence result implies the $\Gamma-$convergence of the associated energy functionals.

We want to stress that the results presented in this work are new even in the linear case that corresponds to $p=2$.

\subsection*{Organization of the paper}
The present article is organized as follows. In Section \ref{sec.intro} we introduce the preliminaries on fractional Sobolev spaces needed in this work. Section \ref{sec.divcurl} is devoted to prove the nonlocal div-curl Lemma. In Section \ref{sec.homog}, we deal with the $H-$convergence compactness result for nonlocal operators via Tartar's method and finally in Section \ref{sec.Gamma} we prove the $\Gamma-$convergence of the associated energy functionals. At the end of the article, we have included an appendix with an abstract compactness result for (nonlinear) monotone operators that is needed in the course of the proof of our main result.

\section{Preliminaries and notation} \label{sec.intro}

In this section we review the basics of fractional order Sobolev spaces. Anyone that is already familiar with nonlocal elliptic-type problems can safely skip this section and return to it only if necessary. 

Here, in order to make the paper self-contained, we only introduce the definitions and results needed in this work. We refer the interested reader to the excellent literature on the subject for a throughout description of these spaces. See for instance the books \cite{Adams, Demengel} and the review article \cite{NPV12}. 

\subsection{The spaces $W^{s,p}(\R^N)$ and $W^{s,p}_0(\O)$} Given $0<s<1\le p <\infty$, the fractional Sobolev space $W^{s,p}(\R^N)$ is defined as
$$
W^{s,p}(\R^N) := \left\{u\in L^p(\R^N)\colon \frac{u(x)-u(y)}{|x-y|^{\frac{N}{p}+s}}\in L^p(\R^N\times\R^N)\right\}.
$$
The norm in this space is then naturally defined as
$$
\|u\|_{s,p} = (\|u\|_p^p + [u]_{s,p}^p)^\frac1p,
$$
where $\|\cdot\|_p$ is, as usual, the $L^p-$norm in $\R^N$ and
$$
[u]_{s,p} := \left(\iint_{\R^N\times\R^N} \frac{|u(x)-u(y)|^p}{|x-y|^{N+sp}}\, dxdy\right)^\frac1p
$$
is the so-called {\em Gagliardo seminorm}.

The space $W^{s,p}(\R^N)$ with the norm $\|\cdot\|_{s,p}$, is a reflexive and separable Banach space. See any of the above mentioned references on fractional order Sobolev spaces for a proof of these facts.

It is also easy to see that smooth functions with compact support are contained in $W^{s,p}(\R^N)$. Also, smooth and rapidly decreasing functions belong to $W^{s,p}(\R^N)$.

Since we need to consider boundary conditions, it is customary to define, given an open set $\O\subset\R^N$, the space of functions that vanish outside $\O$. That is
\begin{equation}\label{Wsp0}
W^{s,p}_0(\O) := \overline{C_c^\infty(\O)} \subset W^{s,p}(\R^N),
\end{equation}
where the closure is taken with respect to the $\|\cdot\|_{s,p}-$norm.

We remark that in some places the following space is considered
$$
\widetilde{W}^{s,p}(\O) := \{u\in W^{s,p}(\R^N)\colon u = 0 \text{ a.e. in } \R^N\setminus\O\}.
$$
Clearly $W^{s,p}_0(\O)\subset \widetilde{W}^{s,p}(\O)$. If the set $\O$ has Lipschitz boundary, both spaces are known to coincide and, moreover, if $sp<1$, $W^{s,p}_0(\Omega)=\widetilde{W}^{s,p}(\O) = \{u|_\O\colon u \in W^{s,p}(\R^N)\}$. See \cite{NPV12}. 

In this article, we always consider $W^{s,p}_0(\O)$ as defined in \eqref{Wsp0}.

The following Poincar\'e type inequalities will be most useful: if $|\O|<\infty$, 
\begin{equation}\label{eq.poincare}
\|u\|_p\le C [u]_{s,p},
\end{equation}
for every $u\in W^{s,p}_0(\O)$, and the constant depends only on $N, s$ and $|\O|$.

The proof of \eqref{eq.poincare} is classical and the reader can find it in any of the above mentioned references.

Observe that from \eqref{eq.poincare}, the Gagliardo seminorm $[\,\cdot\,]_{s,p}$ becomes a norm equivalent to $\|\cdot\|_{s,p}$ in $W^{s,p}_0(\O)$ when $|\O|<\infty$. Hereafter we always use $[\,\cdot\,]_{s,p}$ as the norm in that space.

The extension of the Rellich-Kondrachov compactness theorem to the fractional order Sobolev spaces is also well-known. We state the theorem for future references.

\begin{teo}\label{teo.compacto}
Let $\O\subset \R^N$ be an open set with finite measure. Then the immersion $W^{s,p}_0(\O)\subset L^p(\O)$ is compact. That is, if $\{u_n\}_{n\in\N}\subset W^{s,p}_0(\O)$ is bounded, then there exists $u\in W^{s,p}_0(\O)$ and a subsequence $\{u_{n_k}\}_{k\in\N}\subset \{u_n\}_{n\in\N}$ such that
$$
\|u_{n_k} - u\|_p \to 0 \text{ as } k\to\infty.
$$
\end{teo}

See \cite[Theorem 4.54]{Demengel} for a proof in the case where $\O$ is bounded. The case where $\O$ has finite measure is easily deduced from there.

\subsection{The dual spaces $W^{-s,p'}(\R^N)$ and $W^{-s,p'}(\O)$}
The dual space of $W^{s,p}(\R^N)$ will be denoted by $W^{-s,p'}(\R^N)$. Also, the dual space of $W^{s,p}_0(\O)$ will be denoted by $W^{-s,p'}(\O)$ as usual. Recall that in these spaces the norm is defined as
$$
\|f\|_{-s,p'} := \sup\{\langle f, u\rangle\colon u\in W^{s,p}(\R^N),\ \|u\|_{s,p}=1\}
$$
and
$$
\|f\|_{-s,p',\O} := \sup\{\langle f, u\rangle\colon u\in W^{s,p}_0(\O),\ [u]_{s,p}=1\}.
$$

Observe that $W^{-s,p'}(\R^N)\subset W^{-s,p'}(\O)$ with continuous inclusion.

Given $\{f_n\}_{n\in\N}\subset W^{-s,p'}(\R^N)$ and $f\in W^{-s,p'}(\R^N)$, we say that $f_n\to f$ in $W^{-s,p'}_\text{loc}(\R^N)$ if $\|f_n-f\|_{-s,p',\O}\to 0$ for every $\O\subset \R^N$ bounded and open.

Observe that since $C_c^\infty(\O)\subset W^{s,p}_0(\O)$, the dual space $W^{-s,p'}(\O)$ is contained in the space of distributions $\D'(\O)$.


\subsection{The operator $\L_a$}
Given $0<\lambda < \Lambda<\infty$, we denote by $\A_{\lambda, \Lambda}$ the class
\begin{equation}\label{A}
\A_{\lambda, \Lambda} := \{a\in L^\infty(\R^N\times\R^N) \colon a(x,y)=a(y,x),\  \lambda\le a(x,y) \le \Lambda \text{ a.e.}\}.
\end{equation}

Therefore, for $a\in \A_{\lambda, \Lambda}$ we define the operator $\L_a$ by
\begin{equation}\label{La}
\L_a u (x) = \text{p.v.} \int_{\R^N} a(x,y) \frac{|u(x)-u(y)|^{p-2}(u(x)-u(y))}{|x-y|^{N+sp}}\, dy.
\end{equation}

This operator $\L_a$ is a well defined operator between $W^{s,p}(\R^N)$ and its dual $W^{-s,p'}(\R^N)$ and also between $W^{s,p}_0(\O)$ and $W^{-s,p'}(\O)$. In fact,
\begin{equation}\label{dual}
\langle \L_a u, v\rangle = \frac12\iint_{\R^N\times\R^N} a(x,y)\frac{|u(x)-u(y)|^{p-2}(u(x)-u(y))(v(x)-v(y))}{|x-y|^{N+sp}}\, dxdy.
\end{equation}
The proof of \eqref{dual} is well known. See again \cite{Adams}, for instance. 

In the non symmetric case, one has that
\begin{align*}
\langle \L_a u, v\rangle =& \frac12\iint_{\R^N\times\R^N} a_{\text{sym}}(x,y)\frac{|u(x)-u(y)|^{p-2}(u(x)-u(y))(v(x)-v(y))}{|x-y|^{N+sp}}\, dxdy.
\\&  + \iint_{\R^N\times\R^N} a_{\text{anti}}(x,y) \frac{|u(x)-u(y)|^{p-2}(u(x)-u(y))}{|x-y|^{N+sp}} v(x)\, dxdy,
\end{align*}
where 
$$
a_{\text{sym}}(x,y) = \frac{a(x,y)+a(y,x)}{2}\quad \text{and}\quad a_{\text{anti}}(x,y) = \frac{a(x,y)-a(y,x)}{2},
$$
denote the symmetric and anti-symmetric parts of $a$ respectively.

In order for this operator to be well defined, one needs to impose some extra condition on the anti-symmetric part $a_{\text{anti}}$. For instance,
$$
\sup_{x\in\R^N} \int_{\R^N} \frac{|a_{\text{anti}}(x,y)|^p}{|x-y|^{N+sp}}\, dy <\infty.
$$
See \cite{Felsinger-Kassmann-Voigt, Schilling-Wang} that treat the case in the Hilbert space setting (in our case, that is $p=2$). The extension to general $p\in (1,\infty)$ is straightforward.

In oder to keep the arguments more transparent, we restrict ourselves to the symmetric case.

When $a(x,y)\equiv 1$, the operator $\L_a$ is called, up to some normalization constant, the fractional $p-$laplacian that is denoted by
$$
(-\Delta_p)^s u(x) = C(N,s,p)\text{ p.v.}  \int_{\R^N} \frac{|u(x)-u(y)|^{p-2}(u(x)-u(y))}{|x-y|^{N+sp}}\, dxdy.
$$

\subsection{The Dirichlet problem}
Let $\O\subset \R^N$ be an open set with finite measure and let $a\in \A_{\lambda, \Lambda}$. Given $f\in W^{-s,p'}(\O)$ we define the associated Dirichlet problem as
\begin{equation}\label{DirichletA}
\begin{cases}
\L_a u = f & \text{ in } \O\\
u = 0 & \text{ in } \R^N\setminus \O.
\end{cases}
\end{equation}
We say that $u\in W^{s,p}_0(\O)$ is a weak solution of \eqref{DirichletA} if
$$
\frac12 \iint_{\R^N\times\R^N}  a(x,y)\frac{\pp(u(x)-u(y))(v(x)-v(y))}{|x-y|^{N+sp}}\, dxdy = \langle f, v\rangle,
$$
for every $v\in W^{s,p}_0(\O)$.

Thanks to \eqref{dual}, this is equivalent to say that $\L_a u = f$ in the sense of distributions.

The next proposition is elementary. We include the proof for completeness.
\begin{prop} \label{equivalenciaI} Let $\O\subset \R^N$ be an open set of finite measure, $0<\lambda\le \Lambda<\infty$, $a\in \A_{\lambda, \Lambda}$ and $0<s<1\le p<\infty$ fixed. Then, for any $f\in W^{-s,p'}(\O)$, the following statements are equivalent:
\begin{enumerate}
\item  $u\in W^{s,p}_0(\O), \L_a u=f$ in $\O$, where $\L_a$ is defined by \eqref{DirichletA}.
\item $\J(u)=\min_{v\in W_0^{s,p}(\O)}\J(v)$, where $\J\colon W^{s,p}_0(\O)\to \R$ is defined by
\begin{equation} \label{J}
\J(v)=\frac{1}{2p}\iint_{\R^N\times\R^N}a(x,y) \frac{|v(x)-v(y)|^p}{|x-y|^{N+sp}} \, dxdy-\langle f, v\rangle.
\end{equation}
\end{enumerate}
\end{prop}

\begin{proof}
The proof is standard.

First, we assume (1). Let $v\in W_0^{s,p}(\O)$, and use $u-v$ as a test function in the weak formulation of \eqref{DirichletA} to obtain
\begin{align*}
\frac12 &\iint_{\R^N\times \R^N} a(x,y) \frac{|u(x)-u(y)|^p}{|x-y|^{N+sp}}\, dxdy =\\
& \frac12 \iint_{\R^N\times\R^N}a(x,y) \frac{|u(x)-u(y)|^{p-2}(u(x)-u(y))(v(x)-v(y))}{|x-y|^{N+sp}}\, dxdy +\langle f, u-v\rangle.
\end{align*} 
We now write $a(x,y) = (a(x,y))^{\frac{1}{p}} (a(x,y))^{\frac{1}{p'}}$ and apply Young's inequality to the right-hand-side to obtain
\begin{align*}
\frac12 \iint_{\R^N\times \R^N}& a(x,y) \frac{|u(x)-u(y)|^p}{|x-y|^{N+sp}}\, dxdy \le\\
& \J(v) + \frac{1}{2p'} \iint_{\R^N\times\R^N}a(x,y) \frac{|u(x)-u(y)|^p}{|x-y|^{N+sp}}\, dxdy  +\langle f, u\rangle,
\end{align*} 
from where it follows that $\J(u)\le \J(v)$ for every $v\in W_0^{s,p}(\O)$, which proves (2).

Conversly, now assume (2). Let $t\in \R, v\in W_0^{s,p}(\O)$ and consider $j(t)=\J(u+tv)$. Then, $j$ attains its minimum at $t=0$. Therefore, $0=j'(0)$. That is,
$$
0=\frac12 \iint_{\R^N\times\R^N} a(x,y) \frac{|u(x)-u(y)|^{p-2}(u(x)-u(y))(v(x)-v(y))}{|x-y|^{N+sp}} \, dxdy -\langle f, v\rangle.
$$ 
So, $u$ is the weak solution of \eqref{DirichletA}.
\end{proof}

\begin{prop}\label{prop.min} Let $\O\subset \R^N$ be an open set with finite measure, $0<\lambda\le \Lambda<\infty$, $a\in \A_{\lambda, \Lambda}$ and $0<s<1\le p<\infty$ fixed. Then, for any $f\in W^{-s,p'}(\O)$, there exists a unique $u\in W_0^{s,p}(\O)$ minimizer of $\J$ over $W_0^{s,p}(\O)$, where $\J$ is defined by \eqref{J}.
\end{prop}
\begin{proof}
Clearly, $m:=\inf_{W_0^{s,p}(\O)}\J<+\infty$. We will prove $\J$ is bounded from below.
$$
\J(v)\ge \lambda[v]_{s,p}^p-\|f\|_{-s,p'}[v]_{s,p}\ge(\lambda-\frac{\ve}{p})[v]_{s,p}^p-\frac{C(\ve)}{p'}\|f\|_{-s,p'}^{p'}. 
$$
Choose $0<\ve<p\lambda$, thus, $m\neq -\infty$.

Let $\{u_n\}_{n\in \N} \subset W_0^{s,p}(\O)$ be such that $J(u_n)\to m$, as $n\to\infty$. By the previous inequality, we deduce that $\{u_n\}_{n\in\N}\subset W_0^{s,p}(\O)$ is bounded. Then, since $W^{s,p}_0(\O)$ is a reflexive space, thanks to Alaoglu's theorem, up to a subsequence, there exists $u\in W_0^{s,p}(\O)$ such that $u_n\cd u$ weakly in $W_0^{s,p}(\O)$. Thus, by the weak lower semicontinuity of $\J$ (recall that $\J$ is convex), we obtain
$$
\J(u)\le \liminf_{n\to\infty} \J(u_n)=m=\inf_{W_0^{s,p}(\O)}\J.
$$
The uniqueness of the minimizer follows by the strict convexity of $\J$. Suppose $m=\J(u)=\J(v)$, $u\neq v$. Then, $m\le \J(\frac{u+v}{2})<\frac{\J(u)}{2}+\frac{\J(v)}{2}=m$, which is a contradiction.
\end{proof}

Propositions \ref{equivalenciaI} and \ref{prop.min} trivially imply the following.

\begin{corol}
 Let $\O\subset \R^N$ be an open set with finite measure, $0<\lambda\le \Lambda<\infty$, $a\in \A_{\lambda, \Lambda}$ and $0<s<1\le p<\infty$ fixed. Then, for any $f\in W^{-s,p'}(\O)$, there exists a unique weak solution $u\in W^{s,p}_0(\O)$ to \eqref{DirichletA}.
\end{corol}

\section{A nonlocal div-curl Lemma} \label{sec.divcurl}

In this section we prove a nonlocal version of the div-curl Lemma. This will be a fundamental tool in order to use Tartar's method in homogenization.  In the classical setting this lemma was proved by Tartar in \cite{T78,T77}. Here we do not need the lemma in its full generality. We prove only a special case that will suffices for our purposes. See \cite{Al} where a similar approach is made in the classical setting.

We need to introduce some notation and terminology. Given $u\in W^{s,p}(\R^N)$, we define its $(s,p)-$gradient as
\begin{equation}\label{s-grad}
D_{s,p} u(x,y) := \frac{u(x)-u(y)}{|x-y|^{\frac{N}{p}+s}}.
\end{equation}
Observe that, for any $u\in W^{s,p}(\R^N)$, $D_{s,p} u\in L^p(\R^N\times\R^N)$ and so $|D_{s,p} u|^{p-2} D_{s,p} u\in L^{p'}(\R^N\times\R^N)$.

Now, given $\phi\in L^{p'}(\R^N\times \R^N)$, we define its $(s,p)-$divergence as
\begin{equation}\label{s-div}
d_{s,p}\phi(x) := \text{p.v.} \int_{\R^N} \frac{\phi(x,y) - \phi(y,x)}{|x-y|^{\frac{N}{p}+s}}\, dy. 
\end{equation}

With this definitions we have $(-\Delta_p)^s u= \frac{C(N,s,p)}{2} d_{s,p}(|D_{s,p} u|^{p-2} D_{s,p} u)$. Moreover, if $\L_a$ is given by \eqref{La}, we have $\L_a u = \frac12 d_{s,p}(a |D_{s,p} u|^{p-2} D_{s,p} u)$.

We now need to check that this $(s,p)-$divergence operator is a well defined operator between $L^{p'}(\R^N\times\R^N)$ and $W^{-s,p'}(\R^N)$ and that the following {\em integration by parts formula} holds
\begin{equation}\label{partes}
\iint_{\R^N\times\R^N} \phi D_{s,p}u\, dxdy = \langle d_{s,p}\phi, u\rangle,
\end{equation}
for every $u\in W^{s,p}(\R^N)$ and $\phi\in L^{p'}(\R^N\times \R^N)$.

In order to keep the computations as simple as possible, the following notations will be used: for $\phi \in L^{p'}(\R^N \times \R^N)$ we denote
\begin{align}
\label{notation1}\phi = \phi(x,y);\\
\label{notation2}\phi' = \phi(y,x);\\
\label{notation3}\bar\phi = \phi(x,x).
\end{align}

\begin{teo}\label{teo.partes}
Given $\phi\in L^{p'}(\R^N\times\R^N)$, it follows that $d_{s,p}\phi\in W^{-s,p'}(\R^N)$, where $d_{s,p}\phi$ is defined in \eqref{s-div}. Moreover, for any $u\in W^{s,p}(\R^N)$ the integration by parts formula \eqref{partes} holds true.
\end{teo}

\begin{proof}
Let us define
$$
d_{s,p}^\ve \phi(x) := \int_{|x-y|\ge \ve} \frac{\phi(x,y) - \phi(y,x)}{|x-y|^{\frac{N}{p}+s}}\, dy.
$$
Then, it is easy to see that $d_{s,p}^\ve\phi\in L^{p'}(\R^N)$. In fact,
\begin{align*}
|d_{s,p}^\ve\phi(x)| &\le \int_{|x-y|\ge\ve} \frac{|\phi| + |\phi'|}{|x-y|^{\frac{N}{p}+s}}\, dy\\
&\le \left(\int_{|x-y|\ge \ve} \frac{1}{|x-y|^{N+sp}}\, dy\right)^\frac1p \left(\int_{\R^N} (|\phi| + |\phi'|)^{p'}\, dy\right)^\frac{1}{p'}\\
&= \left(\frac{N\omega_N}{sp \ve^{sp}}\right)^\frac{1}{p}  \left(\int_{\R^N} (|\phi| + |\phi'|)^{p'}\, dy\right)^\frac{1}{p'}.
\end{align*}
From this estimate, one immediately obtain
$$
\|d_{s,p}^\ve \phi\|_{p'}\le 2^{\frac{1}{p'}}\left(\frac{N\omega_N}{sp \ve^{sp}}\right)^\frac{1}{p} \|\phi\|_{p'}.
$$

So $d_{s,p}^\ve\phi\in L^{p'}(\R^N)\subset W^{-s,p'}(\R^N)$, therefore
\begin{align*}
\langle d_{s,p}^\ve\phi, u\rangle &= \int_{\R^N} d_{s,p}^\ve\phi u\, dx\\
&= \int_{\R^N} \int_{|x-y|\ge\ve} \frac{\phi - \phi'}{|x-y|^{\frac{N}{p}+s}} u(x)\, dy\, dx\\
&=  \int_{\R^N} \int_{|x-y|\ge\ve} \phi \frac{u(x)}{|x-y|^{\frac{N}{p}+s}}\, dy\, dx -  \int_{\R^N} \int_{|x-y|\ge\ve} \phi' \frac{u(x)}{|x-y|^{\frac{N}{p}+s}}\, dy\, dx\\
&= \int_{\R^N} \int_{|x-y|\ge\ve} \phi \frac{u(x)}{|x-y|^{\frac{N}{p}+s}}\, dy\, dx -  \int_{\R^N} \int_{|x-y|\ge\ve} \phi \frac{u(y)}{|x-y|^{\frac{N}{p}+s}}\, dy\, dx\\
&= \int_{\R^N} \int_{|x-y|\ge\ve} \phi(x,y) D_{s,p}u(x,y)\, dy\, dx.
\end{align*}
Now we take the limit $\ve\downarrow 0$ and obtain the desired result.
\end{proof}

The next lemma is a crucial step.
\begin{lema} \label{producto} 
Let $\phi_n, \phi_0\in L^{p'}(\R^N\times\R^N)$ be such that $\phi_n\cd\phi_0$ weakly in $L^{p'}(\R^N\times\R^N)$. Assume moreover that $d_{s,p}\phi_n\to d_{s,p}\phi_0$ strongly in $W^{-s,p'}_\text{loc}(\R^N)$. Then, for every $\vp\in W^{1,\infty}(\R^N\times\R^N)$, it follows that $d_{s,p}(\vp\phi_n)\to d_{s,p}(\vp\phi_0)$ strongly in $W^{-s,p'}_\text{loc}(\R^N)$.
\end{lema}

\begin{proof}
In the proof the notations \eqref{notation1}--\eqref{notation3} will be used.

Observe, to begin with, that 
\begin{align*}
d_{s,p}(\vp\phi_n) &= \text{p.v.}\int_{\R^N} \frac{\vp\phi_n - \vp'\phi_n'}{|x-y|^{\frac{N}{p}+s}}\, dy\\
&= \bar\vp\, d_{s,p}\phi_n + \text{p.v.}\int_{\R^N}\left( \frac{\vp-\bar\vp}{|x-y|^{\frac{N}{p}+s}}\phi_n + \frac{\bar\vp - \vp'}{|x-y|^{\frac{N}{p}+s}}\phi_n'\right)\, dy,
\end{align*}
for any $n\ge 0$. Clearly, one has
$$
\bar\vp\, d_{s,p}\phi_n\to \bar\vp\, d_{s,p}\phi_0 \text{ strongly in } W^{-s,p'}_\text{loc}(\R^N).
$$
We now denote, for $n\ge 0$, 
\begin{align*}
& J_n^1 := \text{p.v.}\int_{\R^N} \frac{\vp-\bar\vp}{|x-y|^{\frac{N}{p}+s}}\phi_n\, dy,\\
& J_n^2 := \text{p.v.}\int_{\R^N} \frac{\bar\vp - \vp'}{|x-y|^{\frac{N}{p}+s}}\phi_n'\, dy.
\end{align*}
From Theorem \ref{teo.compacto}, the lemma will be proved if we show that
$$
J_n^i\cd J_0^i \text{ weakly in } L^{p'}_{\text{loc}}(\R^N),\ i=1,2.
$$
We prove this fact for $i=1$, the other case is analogous.

Let $v\in L^p_{\text{loc}}(\R^N)$ and $K\subset \R^N$ compact, so
$$
\int_K J_n^1 v\, dx = \int_{\R^N} J_n^1 v_K\, dx = \iint_{\R^N\times\R^N} \phi_n \frac{\vp-\bar\vp}{|x-y|^{\frac{N}{p}+s}} v_K(x)\, dxdy,
$$
where $v_K = v\chi_K$. Therefore, it suffices to show that $\frac{\vp-\bar\vp}{|x-y|^{\frac{N}{p}+s}} v_K(x)\in L^p(\R^N\times\R^N)$. But,
\begin{align*}
\iint_{\R^N\times\R^N} |v_K(x)|^p \frac{|\vp-\bar\vp|^p}{|x-y|^{N+sp}}\, dxdy &= \int_K |v(x)|^p\left(\int_{\R^N} \frac{|\vp(x,y)-\vp(x,x)|^p}{|x-y|^{N+sp}}\, dy\right)\, dx
\end{align*}
and
\begin{align*}
\int_{\R^N} \frac{|\vp(x,y)-\vp(x,x)|^p}{|x-y|^{N+sp}}\, dy &= \left(\int_{|x-y|<1} + \int_{|x-y|\ge1}\right) \frac{|\vp(x,y)-\vp(x,x)|^p}{|x-y|^{N+sp}}\, dy\\
&= I + II.
\end{align*}
For $I$ observe that $|\vp(x,y)-\vp(x,x)|\le \|\nabla\vp\|_\infty |x-y|$ and so
$$
I\le \|\nabla\vp\|_\infty^p \int_{|x-y|<1} \frac{1}{|x-y|^{N+sp-p}}\, dy = \frac{N\omega_N}{p(1-s)}\|\nabla \vp\|_\infty^p.
$$
Finally, for $II$,
$$
II\le 2^p\|\vp\|_\infty^p \int_{|x-y|\ge 1} \frac{1}{|x-y|^{N+sp}}\, dy = \frac{2^pN\omega_N}{sp}\|\vp\|_\infty^p.
$$

This completes the proof of the lemma.
\end{proof}

Now we are in position to prove the main result of the section.
\begin{lema}[Nonlocal Div-Curl Lemma]\label{divcurl} Let $\phi_n,\phi_0\in L^{p'}(\R^N\times\R^N)$ and let $v_n, v_0\in W^{s,p}(\R^N)$ be such that
$$
\begin{cases}
v_n\cd v_0 &\text{weakly in } W^{s,p}(\R^N),\\
\phi_n\cd \phi_0 &\text{weakly in } L^{p'}(\R^N\times \R^N),\\
d_{s,p} \phi_n \to d_{s,p} \phi_0 &\text{strongly in } W^{-s,p'}_\text{loc}(\R^N).
\end{cases}
$$
Then, $\phi_n D_{s,p} v_n \to \phi_0 D_{s,p} v_0$ in the sense of distributions.
\end{lema}

\begin{remark}
In this special version of the div-curl Lemma, we are considering $\psi_n = D_{s,p} v_n$. In this case, since $\psi_n$ are $(s,p)-$gradients of scalar functions, there is no need for the introduction of the $(s,p)-$curl operator. 
\end{remark}

\begin{proof}
The proof is an easy consequence of the previous lemma. In fact, if $\vp\in C^\infty_c(\R^N \times \R^N)$, from Lemma \ref{producto} and the integration by parts formula \eqref{partes} we get
\begin{align*}
\lim_{n\to\infty} \iint_{\R^N\times \R^N} \phi_n D_{s,p} v_n\vp\, dxdy &= \lim_{n\to\infty} \langle d_{s,p}(\vp\phi_n), v_n\rangle\\
&= \langle d_{s,p}(\vp\phi_0), v_0\rangle\\
&= \iint_{\R^N\times \R^N} \phi_0 D_{s,p} v_0\vp\, dxdy.
\end{align*}
The proof is complete.
\end{proof}

\section{$H-$convergence for nonlocal operators} \label{sec.homog}

Now, let $\{a_n\}_{n\in\N}\subset \A_{\lambda, \Lambda}$ be a sequence of positive and bounded kernels and let $\O\subset \R^N$ be an open set with finite measure. We denote the associated nonlocal operators $\L_n := \L_{a_n}$, given by \eqref{La}.

Now, given $f\in W^{-s,p'}(\O)$ we denote by $u_n\in W^{s,p}_0(\O)$ the unique weak solution to
\begin{equation}\label{eq.n}
\begin{cases}
\L_n u_n = f & \text{ in }\O\\
u_n = 0 & \text{ in } \R^N\setminus \O.
\end{cases}
\end{equation}

Our goal is to show that there exists a subsequence (that we still denote by $\{u_n\}_{n\in\N}$), a function $u_0\in W^{s,p}_0(\O)$ and a positive bounded kernel $a_0\in \A_{\lambda_0, \Lambda_0}$ such that
$$
u_n\cd u_0 \quad\text{weakly in } W^{s,p}_0(\O)
$$
and $u_0$ is a weak solution to
\begin{equation}\label{eq.0}
\begin{cases}
\L_0 u_0 = f & \text{ in }\O\\
u_0 = 0 & \text{ in } \R^N\setminus \O,
\end{cases}
\end{equation}
where $\L_0 u := \L_{a_0}u$.

This is the content of the definition of $H-$convergence.
\begin{defi} For any $n\ge 0$ let $0<\lambda_n\le \Lambda_n<\infty$ and let $a_n\in  \A_{\lambda_n, \Lambda_n}$ be a sequence of kernels. Let us denote by $\L_n$, $n\ge 0$, the associated nonlocal operators given by \eqref{La} with $a=a_n$ respectively.

 We say that the sequence $\{\L_n\}_{n\in\N}$ $H-$converges to $\L_0$, if for any $f\in W^{s,p'}(\Omega)$, the sequence of solutions $\{u_n\}_{n\in\N}$ of 
\begin{equation*}
\begin{cases}
\L_n u_n = f & \text{ in }\O\\
u_n = 0 & \text{ in } \R^N\setminus \O.
\end{cases}
\end{equation*}
satisfies
\begin{align*}
 u_n &\cd u_0 & \mbox{weakly in } W^{s,p}_0(\Omega)\\
 a_n |D_{s,p} u_n|^{p-2} D_{s,p} u_n &\cd  a_0 |D_{s,p} u_0|^{p-2} D_{s,p} u_0 & \mbox{weakly in } L^{p'}(\Omega)
\end{align*}
where $u_0$ is the solution of
\begin{equation*}
\begin{cases}
\L_0 u_0 = f & \text{ in }\O\\
u_0 = 0 & \text{ in } \R^N\setminus \O.
\end{cases}
\end{equation*}
\end{defi}

As we said in the introduction, this notion of convergence was introduced by Murat and Tartar in \cite{Murat-Tartar} generalizing the notion of $G-$convergences for symmetric operators given by Spagnolo in \cite{Sp68, Sp76} and De Giorgi and Spagnolo in \cite{SD73}. All of the above mentioned papers work in the context of linear elliptic PDEs. 

As far as we know, this is the first time that this notion is applied to the nonlocal context.

We start with a couple of simple lemmas.
\begin{lema} \label{cota.1}
Let $\{u_n\}_{n\in\N}\subset W^{s,p}_0(\O)$ be the sequence of weak solutions to \eqref{eq.n}. Then $\{u_n\}_{n\in\N}$ is bounded in $W^{s,p}_0(\O)$ and therefore, up to some subsequence, there exists $u_0\in W^{s,p}_0(\O)$ such that $u_n\cd u_0$ weakly in $W^{s,p}_0(\O)$.
\end{lema}

\begin{proof}
The proof is straightforward. In fact, from the properties of the kernel $a_n$, we have
\begin{align*}
\lambda [u_n]_{s,p}^p=\lambda \|D_{s,p} u_n\|_p^p &\le \iint_{\R^N\times\R^N} a_n(x,y) |D_{s,p} u_n(x,y)|^p\, dxdy\\
&= 2\langle \L_n u_n, u_n \rangle \\
&= 2\langle f, u_n\rangle\\
&\le 2\|f\|_{-s,p'} \|D_{s,p} u_n\|_p= 2\|f\|_{-s,p'} [u_n]_{s,p}.
\end{align*}
Therefore
$$
[u_n]_{p,s}\le (2\lambda^{-1} \|f\|_{-s,p'})^\frac{1}{p-1}.
$$
From this uniform bound, the rest of the lemma follows.
\end{proof}

\begin{lema}\label{cota.2}
Let $\{u_n\}_{n\in\N}\subset W^{s,p}_0(\O)$ be the sequence of weak solutions to \eqref{eq.n}. Then the sequence of fluxes $\{\xi_n:=a_n |D_{s,p} u_n|^{p-2} D_{s,p} u_n\}_{n\in\N}\subset L^{p'}(\R^N\times \R^N)$ is bounded and therefore, up to some subsequence, there exists $\xi_0\in L^{p'}(\R^N\times\R^N)$ such that $\xi_n\cd \xi_0$ weakly in $L^{p'}(\R^N\times\R^N)$.
\end{lema}

\begin{proof}
The proof is also straightforward. In fact, from the boundedness of the kernels $\{a_n\}_{n\in\N}$ and from Lemma \ref{cota.1}, we have
\begin{align*}
\iint_{\R^N\times\R^N} |\xi_n|^{p'}\, dxdy &=  \iint_{\R^N\times\R^N} |a_n |D_{s,p} u_n|^{p-2} D_{s,p} u_n|^{p'}\, dxdy \\
&\le \Lambda^{p'}  \iint_{\R^N\times\R^N} |D_{s,p} u_n|^p\, dxdy\\
&\le (2\Lambda \lambda^{-1})^{p'} \|f\|_{-s,p'}^{p'}.
\end{align*}
The proof is complete.
\end{proof}

The following observation is trivial.
\begin{prop} \label{prop.mono}
 The sequence of operators $\{\L_n\}_{n\in\N}$ is uniformly strictly monotone.
\end{prop}

\begin{proof}
The proof follows immediately from the inequality
\begin{equation}\label{simon}
(|a|^{p-2}a - |b|^{p-2}b)(a-b)\ge \begin{cases}
c_p |a-b|^p & \text{if } p\ge 2\\
c_p \frac{|a-b|^2}{(|a|+|b|)^{2-p}} & \text{if } 1<p<2,
\end{cases}
\end{equation}
for every $a, b\in \R$, with $c_p>0$ depending only on $p$ (see \cite{Simon}) and from the uniform estimate $\lambda\le a_n(x,y)\le \Lambda$ a.e. $(x,y)\in \R^N\times\R^N$.
\end{proof}

The oscillating test function method of Tartar needs the existence of such test functions. This is the content of the next lemma.
\begin{lema} \label{f.aux}
Given a sequence $\{a_n\}_{n\in\N}\subset \mathcal A_{\lambda,\Lambda}$ and a function $w_0\in W^{s,p}(\R^N)$, there exists a sequence $\{w_n\}_{n\in\N}\subset W^{s,p}(\R^N)$ and $g_0\in W^{-s,p'}(\R^N)$ such that
\begin{align}\label{f.w}
w_n \cd w_0 &\quad \mbox{ weakly in } W^{s,p}(\R^N)\\
g_n:= \L_n w_n \to g_0 &\quad \mbox{ strongly in } W^{-s,p'}_\text{loc}(\R^N).
\end{align}
\end{lema}

\begin{proof}
First, observe that the operators $\L_n\colon W^{s,p}(\R^N)\to  W^{-s,p'}(\R^N)$ verify the following estimates:
\begin{align}
\label{5.3}&\|\L_n u\|_{-s,p'}\le \frac{\Lambda}{2} [ u]_{s,p}^\frac{p}{p'},\\
\label{5.4}&\langle \L_n u, u\rangle \ge \frac{\lambda}{2} [u]_{s,p}^p.
\end{align}
These estimates follow easily from the definitions and H\"older's inequality.

Now, we define the operator $\hat{\L}_n \colon W^{s,p}(\R^N)\to W^{-s,p'}(\R^N)$ by $\hat{\L}_n u = \L_n u + |u|^{p-2} u$. From \eqref{5.3} and \eqref{5.4}, it follows that $\hat{\L}_n$ verifies the estimates
\begin{align}
\label{5.3'}& \|\hat{\L}_n u\|_{-s,p'}\le \max\left\{\frac{\Lambda}{2};1 \right\} \|u\|_{s,p}^{\frac{p}{p'}},\\
\label{5.4'}& \langle \hat{\L}_n u, u\rangle \ge \min\left\{\frac{\lambda}{2};1 \right\} \|u\|_{s,p}^p. 
\end{align}

Proposition \ref{prop.mono} implies the monotonicity of $\hat{\L}_n$. Observe that $\hat{\L}_n$ is continuous on finite-dimensional subspaces of $W^{s,p}(\R^N)$, therefore, by \cite[Corollary 17.2]{chipot}, $\hat{\L}_n$ admits an inverse, $\hat{\L}_n^{-1}$.

Let us check that the family of operators $\{\hat{\L}^{-1}_n\}_{n\in\N}$ fulfills the hypotheses of Theorem \ref{lema.conv}.  The operators $\hat{\L}^{-1}_n$ are uniformly strictly monotone since are the inverse of the sequence of uniformly strictly monotone operators $\{\hat{\L}_n\}_{n\in\N}$.  

Observe that from \eqref{5.3'} and \eqref{5.4'} one immediately obtains
\begin{equation}\label{5.6}
\langle \hat{\L}_n u, u\rangle \ge c \|\hat{\L}_n u\|_{-s,p'}^{p'},
\end{equation}
where $c:=\frac{\min\left\{\frac{\lambda}{2};1 \right\}}{\left(\min\left\{\frac{\Lambda}{2};1 \right\}\right)^{p'}}=c(\lambda,\Lambda,p')$, which can be written as
$$
\langle f, \hat{\L}_n^{-1}f \rangle \ge c \|f\|_{-s,p'}^{p'}
$$
for every  $f\in W^{-s,p'}(\R^N)$, from where we get
$$
	c \|f\|_{-s,p'}^\frac{1}{p-1} \leq \frac{\langle f, \hat{\L}_n^{-1}f \rangle}{\|f\|_{-s,p'}}  \to \infty
$$
as $\|f\|_{-s,p'}$ approaches infinity. Consequently,  $\{\hat{\L}_n^{-1}\}_{n\in\N}$ is uniformly coercive.

From \eqref{5.4'} it follows that
$$
c \|u\|_{s,p}^p \leq \|\hat{\L}_n u\|_{-s,p'}\|u\|_{s,p}
$$ 
where $c=\min\left\{\frac{\lambda}{2};1 \right\}$, that is, 
$$
 \|\hat{\L}_n^{-1}f\|_{s,p}^{p-1} \leq c^{-1} \|f\|_{-s,p'}.
$$ 
Since $c$ is independent on $n$, it follows that $\sup_{n\in\N} \|\hat{\L}_n^{-1}f\|_{s,p} <\infty$.

It remains to prove that $\{\hat{\L}_n^{-1}\}_{n\in\N}$ is uniformly strong-weak continuous, but this is a consequence of the fact that these operators are uniformly strong-strong continuous. In fact, let $f, g\in W^{-s,p'}(\R^N)$ and let $u_n = \hat{\L}_n^{-1} f$ and $v_n = \hat{\L}_n^{-1} g$. Now, if $p\ge 2$, from \eqref{simon} it follows that, calling $w_n = u_n-v_n$,
\begin{align*}
\lambda c_p [w_n]_{s,p}^p &\le \iint_{\R^N\times\R^N} a_n \frac{|u_n(x)-u_n(y)|^{p-2}(u_n(x)-u_n(y))(w_n(x) - w_n(y))}{|x-y|^{N+sp}}\, dxdy\\
& -\iint_{\R^N\times\R^N} a_n \frac{|v_n(x)-v_n(y)|^{p-2}(v_n(x)-v_n(y))(w_n(x) - w_n(y))}{|x-y|^{N+sp}}\, dxdy.
\end{align*}
On the other hand,
$$
c_p \|w_n\|_p^p \le \int_{\R^N} (|u_n|^{p-2}u_n - |v_n|^{p-2}v_n) w_n\, dx.
$$
Therefore, adding up these two inequalities, we get
$$
\|w_n\|_{s,p}^p\le C \langle f-g, w_n\rangle \le C \|f-g\|_{-s,p'} \|w_n\|_{s,p},
$$
where $C$ depends only on $p$ and $\lambda$. This completes the claim for the case $p\ge 2$. 

Finally, if $1<p<2$, we observe that, from H\"older's inequality,
\begin{align*}
[w_n]_{s,p}^p \le& \left(\iint_{\R^N\times\R^N}  \frac{1}{|x-y|^{N+sp}}\frac{|w_n(x)-w_n(y)|^2}{(|u_n(x)-u_n(y)|+|v_n(x)-v_n(y)|)^{2-p}}\, dxdy\right)^{\frac{p}{2}}\\
&\times \left(\iint_{\R^N\times\R^N} \frac{(|u_n(x)-u_n(y)|+|v_n(x)-v_n(y)|)^p}{|x-y|^{N+sp}}\, dxdy\right)^{\frac{2-p}{2}}.
\end{align*}
In an analogous manner,
$$
\|w_n\|_p^p \le \left(\int_{\R^N} \frac{|w_n|^2}{(|u_n|+|v_n|)^{2-p}}\, dx\right)^{\frac{p}{2}} \left(\int_{\R^N} (|u_n|+|v_n|)^p\, dx\right)^{\frac{2-p}{2}}.
$$
From these two inequalities, reasoning exactly as in the previous case, one can conclude the uniform strong-strong continuity for the case $1<p<2$.

Then, by Theorem \ref{lema.conv}, there exists a subsequence of operators, that we still denote by $\{\hat{\L}_n^{-1}\}_{n\in\N}$, and an strong-weak continuous, uniformly coercive, uniformly strictly monotone operator $\hat{\L}_0^{-1}\colon W^{-s,p'}(\R^N)\to W^{s,p}(\R^N)$, such that 
\begin{equation}\label{convergen.inversos}
\hat\L_n^{-1} f\cd \hat\L_0^{-1}f \quad \text{ weakly in } W^{s,p}(\R^N) \text{ for every } f\in W^{-s,p'}(\R^N).
\end{equation}

Since $\hat{\L}_0^{-1}$ is continuous on finite subspaces of $W^{s,-p}(\R^N)$, again,  by \cite[Corollary 17.2]{chipot}, $\hat{\L}_0^{-1}$ is invertible, that is, there exists a linear continuous operator $\hat{\L}_0 \colon W^{s,p}(\R^N) \to W^{-s,p'}(\R^N)$. Observe that $\hat{\L}_0$ satisfies \eqref{5.3'} and \eqref{5.4'}.

Consider $\hat{g}_0:=\hat{\L}_0 w_0 \in W^{-s,p'}(\R^N)$ and define $w_n := \hat{\L}_n^{-1} \hat{g}_0\in W^{s,p}(\R^N)$.  Thus, by \eqref{convergen.inversos} we obtain that $w_n\cd w_0$ in $W^{s,p}(\R^N)$.

Finally, if we denote $g_n := \L_n w_n$, we obtain that 
$$
g_n = \L_n w_n = \hat{\L}_n w_n - |w_n|^{p-2}w_n = \hat{g}_0 - |w_n|^{p-2} w_n.
$$
Since $w_n\cd w_0$ weakly in $W^{s,p}(\R^N)$ it follows that $w_n\to w_0$ strongly in $L^p_\text{loc}(\R^N)$, therefore
$$
g_n \to \hat{g}_0 - |w_0|^{p-2} w_0 =: g_0 \text{ strongly in } W^{-s,p'}_\text{loc}(\R^N).
$$
The proof is complete.
\end{proof}

With all of these preliminaries, we are ready to prove the main result of the paper.
\begin{teo}\label{main}
Let $\O\subset \R^N$ be an open set with finite measure and let $0<\lambda\le \Lambda <\infty$. Then, for any sequence $\{a_n\}_{n\in\N}\subset \A_{\lambda, \Lambda}$, there exists subsequence $\{a_{n_k}\}_{k\in\N}\subset \{a_n\}_{n\in\N}$ and a kernel $a_0\in \A_{\lambda, \frac{\Lambda^{p'}}{\lambda}}$ such that the sequence of operators $\{\L_{n_k}\}_{k\in\N}$, $H-$converges to $\L_0$.
\end{teo}

\begin{proof}
Consider $w_0(x)=e^{-|x|^2} \in W^{s,p}(\R^N)$ and let $\{w_n\}_{n\in\N}\subset W^{s,p}(\R^N)$ be the sequence given by Lemma \ref{f.aux}.

Let us denote by $\eta_n=a_n|D_{s,p} w_n|^{p-2} D_{s,p}w_n$ and observe that from \eqref{f.w} and the boundedness of the kernels $a_n$ it follows that
$$
\|\eta_n\|_{p'} \leq \Lambda \|D_{s,p} w_n\|_p^{\frac{p}{p'}} = \Lambda [w_n]_{s,p}^{\frac{p}{p'}}\le C.
$$			
Then, there exists a function $\eta_0 \in L^{p'}(\R^N\times \R^N)$ such that, up to a subsequence,
$$
\eta_n \cd \eta_0 \quad\mbox{weakly in } L^{p'}(\R^N\times \R^N).
$$
Given $\theta\in \R$, we apply Lemma \ref{divcurl} to the following nonnegative quantity
$$
(\xi_n - |\theta|^{p-2}\theta \eta_n) (D_{s,p} u_n - \theta D_{s,p} w_n)  \ge 0,
$$
where, as in Lemma \ref{cota.2}, we note $\xi_n(x,y) = a_n(x,y)|D_{s,p} u_n(x,y)|^{p-2} D_{s,p} u_n(x,y)$.

Therefore,
\begin{equation} \label{dc}
(\xi_n - |\theta|^{p-2}\theta \eta_n) (D_{s,p} u_n - \theta D_{s,p} w_n)\to (\xi_0 - |\theta|^{p-2}\theta \eta_0) (D_{s,p} u_0 - \theta D_{s,p} w_0)\ge 0,
\end{equation}
in the sense of distributions.

Take now $\theta = \theta_t = \frac{(u_0(x)-u_0(y)) - t\theta_0}{w_0(x)-w_0(y)}$, where $\theta_0\in\R$ and $t>0$. Observe that $\theta_t$ is well defined a.e. in $\R^N\times\R^N$. Therefore, by \eqref{dc} we obtain
$$
(\xi_0 - |\theta_t|^{p-2}\theta_t \eta_0)\theta_0 \ge 0.
$$
Since $\theta_0\in\R$ is arbitrary, we conclude that
$$
\xi_0 - |\theta_t|^{p-2}\theta_t \eta_0=0,
$$
for every $t>0$. Passing to the limit $t\downarrow 0$, we get
\begin{equation}\label{xi.eta}
\xi_0 = |\theta_u|^{p-2}\theta_u \eta_0,
\end{equation}
where $\theta_u = \frac{u_0(x)-u_0(y)}{w_0(x)-w_0(y)}$.

Now, we obtain 
\begin{equation}\label{mu0}
\xi_0(x,y) = a_0(x,y) |D_{s,p} u_0(x,y)|^{p-2} D_{s,p} u_0(x,y),
\end{equation}
where $a_0(x,y) :=\frac{\eta_0(x,y)}{|D_{s,p} w_0(x,y)|^{p-2} D_{s,p} w_0(x,y)} $.

Finally, observe that from \eqref{eq.n} and Lemma \ref{cota.2}, it follows that
$$
\frac12\iint_{\R^N\times\R^N} \xi_0 D_{s,p} v\, dxdy = \langle f,v\rangle,
$$
for every $v\in W_0^{s,p}(\O)$. But, by \eqref{mu0}
$$
\xi_0 D_{s,p} v =a_0 |D_{s,p} u_0|^{p-2} D_{s,p} u_0 D_{s,p} v,
$$
then, $u_0$ is the weak solution of \eqref{eq.0}.

To conclude the proof of the theorem, it remains to show that $a_0\in \A_{\lambda, \frac{\Lambda^{p'}}{\lambda}}$, but this is the content of Proposition \ref{A0} that we prove next.
\end{proof}

The next proposition shows the coercivity and boundedness of the homogenized kernel $a_0$.
\begin{prop}\label{A0}
Under the same assumptions and notations of Theorem \ref{main}, the homogenized kernel $a_0$ belongs to the class $\A_{\lambda,\frac{\Lambda^{p'}}{\lambda}}$.
\end{prop}

\begin{proof}
First, we prove the boundedness from below  $a_0(x,y) \ge \lambda$, a.e. $x,y\in \R^N$. Fix $v_0 \in W^{s,p}(\R^N)$ (for instance $v_0(x) = e^{-|x|^2}$) and denote by $v_n$ the solution of 
\begin{equation}\label{eq.zn}
\begin{cases}
\L_n v_n = \L_0 v_0 & \text{ in }\O\\
v_n = 0 & \text{ in } \R^N\setminus \O.
\end{cases}
\end{equation}

By Lemma \ref{cota.1}, $\{v_n\}_{n\in \N}$ is bounded in $W_0^{s,p}(\O)$. Then, it has a weak limit in $W_0^{s,p} (\O)$. But, by Theorem \ref{main}, that limit is $v_0$. Applying the nonlocal div-curl Lemma, Lemma \ref{divcurl}, to the sequences $\{a_n |D_{s,p} v_n|^{p-2} D_{s,p}
 v_n\}_{n\in \N}$ and $\{v_n\}_{n\in \N}$, we obtain  
\begin{equation}\label{div.curl.phi}
a_n |D_{s,p} v_n|^p \to a_0 |D_{s,p} v_0|^p,
\end{equation}
in the sense of distributions. 

Since $a_n\in \A_{\lambda, \Lambda}$, 
$$
\lambda \iint_{\R^N\times \R^N} |D_{s,p} v_n|^p\vp\, dxdy \le \iint_{\R^N\times \R^N} a_n |D_{s,p} v_n|^p\vp\, dxdy,
$$
for every $\vp\in C^\infty_c(\R^N\times \R^N)$, $\vp\ge 0$.

Therefore, from \eqref{div.curl.phi} and since the left hand side is weak lower semi-continuous in $L^p(\R^N\times\R^N)$, we obtain
$$
\lambda \iint_{\R^N\times \R^N} |D_{s,p} v_0|^p \vp\, dxdy\le \iint_{\R^N\times \R^N} a_0 |D_{s,p} v_0|^p\vp\, dxdy.
$$

Since $0\le \vp\in C^\infty_c(\R^N\times \R^N)$ is arbitrary, we conclude that
\begin{equation}\label{puntual}
\lambda |D_{s,p} v_0|^p \le a_0 |D_{s,p} v_0|^p, \text{ a.e. in } \R^N\times \R^N.
\end{equation}
Now, observe that \eqref{puntual} holds for any $v_0\in W^{s,p}(\R^N)$ and so
$$
\lambda \le a_0 \quad \text{ a.e. in } \R^N\times \R^N,
$$
as we wanted to prove.

It remains to prove the boundedness from above  $a_0  \le \frac{\Lambda^{p'}}{\lambda} $ a.e. in $\R^N\times \R^N$.

Take $\vp\in C^\infty_c(\R^N\times \R^N)$ be nonnegative and by our hypotheses on the kernel $a_n$ we have
\begin{align*}
\iint_{\R^N\times\R^N} |a_n |D_{s,p} v_n|^{p-2} D_{s,p} v_n|^{p'} \vp\, dxdy &\le \Lambda^{p'} \iint_{\R^N\times \R^N} |D_{s,p} v_n|^p \vp\, dxdy\\
&\le \frac{\Lambda^{p'}}{\lambda} \iint_{\R^N\times \R^N} a_n |D_{s,p} v_n|^p \vp\, dxdy.
\end{align*}

From this point the proof follows as in the previous case, using the convergence of the fluxes $a_n |D_{s,p} v_n|^{p-2} D_{s,p} v_n \cd a_0 |D_{s,p} v_0|^{p-2} D_{s,p} v_0$ weakly in $L^{p'}(\R^N\times \R^N)$.

The proof is now complete.
\end{proof}

\section{Gamma convergence}\label{sec.Gamma}
The purpose of this section is to prove that the notion of $H-$convergence of the functionals associated to \eqref{ec.ii.1} is equivalent to the $\Gamma-$convergence of their associated energy functionals. Our arguments follow  closely the ideas from \cite{BoMa76,Ma73}.

Let us begin by recalling the definition of $\Gamma-$convergence.
\begin{defi}
Let $X$ be a metric space and let $J_n\colon X\to \bar\R$, $n\ge 0$.

We say that $J_n$ $\Gamma-$converges to $J_0$ if the following two inequalities hold
\begin{enumerate}
\item[] (liminf inequality) For every $x\in X$ and every sequence $\{x_n\}_{n\in \N}\subset X$ such that $x_n\to x$, 
$$
J_0(x) \le \liminf_{n\to\infty} J_n(x_n).
$$

\item[](limsup inequality)  For every $x\in X$ there exists a sequence $\{y_n\}_{n\in\N}\subset X$ such that
$$
J_0(x)\ge \limsup_{n\to\infty} J_n(y_n).
$$
\end{enumerate}
\end{defi}

This notion of convergence was introduced by De Giorgi in the 70s (see \cite{DeGiorgi} and \cite{DeGiorgi-Franzoni}) and has been proved to be an extremely useful tool when dealing with the convergence of variational problems. See, for instance the book of Dal Maso \cite{DalMaso} for a throughout description of the $\Gamma-$convergence and its properties and also the book of Braides \cite{Braides} where many different applications of this notion of convergence are shown.

The main feature of this notion of convergence is the fact that minimizers of $J_n$ converges to minimizers of $J_0$. In the case of convex functionals, this notion naturally relates to the notion of Legendre transform for convex proper functionals.

\begin{defi}
Given a convex proper functional $J\colon X \to (-\infty, \infty]$ its Legendre transform $J^*\colon X'\to (-\infty, \infty]$ is defined as
$$
J^*(f):=\sup_{x\in X}\{\langle f,x\rangle - J(x)\}.
$$
\end{defi}

In \cite{BoMa76}, Boccardo and Marcellini relates the $\Gamma-$convergence of convex proper functionals with the pointwise convergence of their Legendre transforms. For the sake of completeness we include an elementary proof.
\begin{prop}[Boccardo and Marcellini, \cite{BoMa76}]\label{gamma}
Let $X$ be a separable and reflexive Banach space and $J_n, J_0\colon X\to (-\infty, \infty]$ be convex proper and weakly lower semicontinuous functionals.

Then, $J_n$ $\Gamma-$converges to $J_0$ if and only if $J_n^*(f)\to J_0^*(f)$ for every $f\in X'$.
\end{prop}

\begin{proof}
Assume first that $J_n$ $\Gamma-$converges to $J_0$

Given $f\in X'$ it is easy to see that $J_n(\cdot) - \langle f, \cdot\rangle$  $\Gamma-$converges to $J_0(\cdot) - \langle f, \cdot\rangle$ (see \cite[Proposition 6.21]{DalMaso}), therefore, the fundamental theorem of $\Gamma-$convergence (\cite[Theorem 7.4]{DalMaso}) claims the convergence of the infima
$$
\lim_{n\to \infty} \inf_{x\in X}\{J_n(x)-\langle f,x \rangle\} = \inf_{x\in X}\{ J_0(x)-\langle f,x \rangle\},
$$
that is, $J_n^*(f) \to J_0^*(f)$ as $n\to\infty$.

Assume now that $J_n^*(f)\to J_0^*(f)$ for any $f\in X'$ as $n\to\infty$.

By \cite[Theorem 8.5]{DalMaso}, there exists a subsequence $\{J_{n_k}\}_{k\in\N}\subset \{J_n\}_{n\in\N}$ and a lower semicontinuous functional $G$ such that $J_{n_k}$ $\Gamma-$converges to $G$. Since $J_n$ is convex for every $n\in\N$, it follows from \cite[Theorem 11.1]{DalMaso} that $G$ is also convex.
	
The first implication above implies that $J_{n_k}^*$ converges pointwise to $G^*$, which in turn implies that $J_0^*=G^*$. Applying the Legendre transform to the previous equality it follows, since $X$ is reflexive, that $J_0=G$. Since the sequence $J_{n_k}$ is arbitrary, the Urysonh property of the $\Gamma-$convergence (\cite[Theorem 8.3]{DalMaso}) implies the desired result.
\end{proof}

Theorem \ref{main} claims that for $0<\lambda\le \Lambda<\infty$ fixed and any sequence $\{a_n\}_{n\in\N} \subset \A_{\lam,\Lambda}$, there exists a kernel $a_0\in \A_{\lam,\frac{\Lambda^{p'}}{\lam}}$ such that $\{\L_n\}_{n\in\N}$ $H-$converges to $\L_0$ up to some subsequence. From now on, we always assume that $\L_n$ $H-$converges to $\L_0$.

The sequence of operators $\{\L_n\}_{n\in\N}$ and the limit operator $\L_0$  define a sequence of energy functionals $\{\J_n\}_{n\in\N}$ and a limit functional $\J_0$, given by
\begin{equation} \label{Jn}
	\J_n(v)=\frac{1}{2p}\iint_{\R^N\times\R^N}a_n(x,y) \frac{|v(x)-v(y)|^p}{|x-y|^{N+sp}} \, dxdy 
\end{equation}
for $n\ge 0$, defined in $W^{s,p}_0(\Omega)$.

We then define, for $n\ge 0$, $J_n\colon L^p(\Omega)\to (-\infty, \infty]$ as
\begin{align}\label{Jotas}
J_n(v) :=
\begin{cases}
	\J_n(v) &\quad \mbox{if } v\in W^{s,p}_0(\Omega)\\
	+\infty&\quad \mbox{ otherwise.}
\end{cases}
\end{align}

Recall that Proposition \ref{equivalenciaI} implies that given $f\in L^{p'}(\Omega)$, $u_n\in L^p(\Omega)$ is a weak solution of 
\begin{equation}\label{eq.B}
\begin{cases}
\L_n u_n = f &\text{in }\Omega\\
u_n=0 & \text{in } \R^N\setminus \Omega,
\end{cases}
\end{equation}
if and only if $u_n$ verifies
$$
J_n(u_n) - \langle f, u_n\rangle = \inf_{v\in L^p(\Omega)} J_n(v) - \langle f, v\rangle.
$$
In other words, $u_n\in L^p(\Omega)$ is a weak solution of \eqref{eq.B} if and only if
$$
J_n^*(f) = \langle f, u_n\rangle - J_n(u_n).
$$

Furthermore, since we have
$$
J_n(u) = \frac{1}{p} \langle \L_n u, u\rangle,
$$
it follows that $u_n\in L^p(\Omega)$ is a weak solution of \eqref{eq.B} if and only if
\begin{equation}\label{maravilla}
J_n^*(f) = \frac{1}{p'} \langle f, u_n\rangle.
\end{equation}

From \eqref{maravilla} is fairly easy to check that if $\L_n$ $H-$converges to $\L_0$, then $J_n^*(f)\to J_0^*(f)$ for every $f\in L^{p'}(\Omega)$.

With the help of Proposition \ref{gamma} all of these considerations imply the $\Gamma-$convergence of the functionals given in \eqref{Jotas}. That is, we have proved the following theorem
\begin{teo}
Let $0<\lambda\le \Lambda<\infty$ and $\{a_n\}_{n\in\N}\subset \A_{\lambda, \Lambda}$. Let $\L_n := \L_{a_n}$ be the operators defined in \eqref{La} and assume that $\L_n$ $H-$converges to $\L_0 = \L_{a_0}$ for some kernel $a_0\in \A_{\lambda, \frac{\Lambda^{p'}}{p}}$.

Then, the associated functionals $J_n$ given by \eqref{Jotas} $\Gamma-$converges to $J_0$.
\end{teo}

\appendix

\section{Compactness results for nonlinear monotone operators}

In this section we prove a compactness result for monotone operators. This results are nonlinear analogues to \cite[Lemmas 1.3.3 and 1.3.4]{Al} and are crucial in the construction of the oscillating test functions (see Lemma \ref{f.aux}).

We recall the following definitions.
\begin{defi}
Let $X$ be a Banach space and let $S\colon X'\to X$. We say that $S$ is coercive if
$$
\frac{\langle f, Sf\rangle}{\|f\|} \to \infty \quad \text{if } \|f\|\to\infty.
$$

If now $S_n\colon X'\to X$, $n\in \N$, we say that $\{S_n\}_{n\in\N}$ is uniformly coercive if
$$
\inf_{n\in\N} \frac{\langle f, S_n f\rangle}{\|f\|} \to \infty  \quad \text{if } \|f\|\to\infty.
$$
\end{defi}

\begin{defi}
Let $X$ be a Banach space and let $S\colon X'\to X$. We say that $S$ is monotone if
$$
\langle f-g, Sf - Sg\rangle \ge 0\quad \text{for every } f, g\in X'.
$$

We say that $S$ is strictly monotone if the equality above only holds when $f=g$.

If now $S_n\colon X'\to X$, $n\in \N$, we say that $\{S_n\}_{n\in\N}$ is uniformly strictly monotone if 
$$
\inf_{n\in\N} \langle f-g, S_n f - S_n g\rangle >0 \quad \text{for every } f, g\in X', f\neq g.
$$
\end{defi}

\begin{defi}
Let $X$ be a Banach space and let $S\colon X'\to X$. We say that $S$ is strong-weak continuous if
$$
Sf_k\cd Sf\quad \text{if } f_k\to f.
$$

If now $S_n\colon X'\to X$, $n\in \N$, we say that $\{S_n\}_{n\in\N}$ is uniformly strong-weak continuous if
$$
\sup_n \langle g, S_n f_k - S_n f\rangle \to 0\quad  \text{if } f_k\to f, \text{ for every } g\in X.
$$
\end{defi}

We now have this compactness result for operators.
\begin{teo} \label{lema.conv}
	Let $X$ be a separable reflexive Banach space. Let $S_n\colon X'\to X$ be a sequence of monotone operators that are uniformly strong-weak continuous and uniformly coercive. Assume that for every $f\in X'$, $\sup_{n\in\N} \|S_n f\|<\infty$. Then there exists a subsequence, still denoted by $\{S_n\}_{n\in\N}$, and a limit operator $S_0$ such that
	$$
		 S_n f\cd S_0 f \quad \mbox{weakly in } X
	$$
	for any $f\in X'$. Moreover, $S_0$ is a uniformly coercive, strong-weak continuous and monotone operator.
	
	Moreover, if $\{S_n\}_{n\in\N}$ is uniformly strictly monotone, then $S_0$ is strictly monotone.
\end{teo}

\begin{proof}
Let $\D$ be a dense countable subset of $X'$. Since $\sup_{n\in\N}\|S_n f\|<\infty$, by a standard diagonal argument, there exists a subsequence, that we still denote by $\{S_n\}_{n\in\N}$ such that
\begin{equation}\label{conv.s}
S_n f \cd S_0 f  \quad \mbox{weakly in } X,
\end{equation}
for every $f\in \D$.
	 
This defines an operator $S_0\colon \D \to X$. Let us first see that $S_0$ can be extended to $X'$ and that $S_n f \cd S_0 f$ for every $f\in X'$. In fact, if $f\in X'$, there exists $\{f_k\}_{k\in\N}\subset \D$ such that $f_k\to f$ strongly in $X'$ and then
$$
\langle g, S_0 f_k - S_0 f_j\rangle = \langle g, S_0 f_k - S_n f_k\rangle + \langle g, S_n f_k - S_n f_j\rangle + \langle g, S_n f_j - S_0 f_j\rangle,
$$
so
\begin{align*}
|\langle g, S_0 f_k - S_0 f_j\rangle| \le& |\langle g, S_0 f_k - S_n f_k\rangle| + | \langle g, S_n f_j - S_0 f_j\rangle|\\
& + \sup_{n\in\N} (|\langle g, S_n f_k - S_n f\rangle| + |\langle g, S_n f_j - S_n f\rangle|)\\
<&  |\langle g, S_0 f_k - S_n f_k\rangle| + | \langle g, S_n f_j - S_0 f_j\rangle| + \ve,
\end{align*}
if $k,j\ge k_0$ by the uniform strong-weak conitnuity of the sequence $\{S_n\}_{n\in\N}$. Taking limit $n\to\infty$ in the right-han-side of the former inequality gives that $\{S_0 f_k\}_{k\in\N}\subset X$ is weakly Cauchy. Therefore, there exists a point, that we denote by $S_0 f\in X$ such that 
$$
S_0 f_k \cd S_0 f \quad \text{weakly in } X.
$$

A completely analogous argument shows that the limit $S_0 f$ is independent of the sequence $\{f_k\}_{k\in\N}\subset \D$ and that $S_n f\cd S_0 f$ weakly in $X$ for every $f\in X'$. Moreover, the exact same argument shows that $S_0\colon X'\to X$ is strong-weak continuous.

The rest of the properties of the limit operator $S_0$ are easily deduced from the convergence $S_n f \cd S_0 f$ weakly in $X$.
\end{proof}

%


\section*{Acknowledgements}
This paper was partially supported by grants UBACyT 20020130100283BA, CONICET PIP 11220150100032CO and ANPCyT PICT 2012-0153. 

J. Fern\'andez Bonder and Ariel M. Salort are members of CONICET and A. Ritorto is a doctoral fellow of CONICET.

\bibliographystyle{amsplain}
\bibliography{biblio}
\end{document}